\documentclass[10pt,reqno]{amsart}

\usepackage{fullpage,amsfonts,amsmath,amssymb}

\input xypic
  \usepackage[all, knot]{xy}
  \xyoption{arc}
\theoremstyle{plain}
\newtheorem*{lem}{Lemma}

\newtheorem*{prop}{Proposition}

\newtheorem*{cor}{Corollary}
\newtheorem*{thrm}{Theorem}

\theoremstyle{Definition}
\newtheorem*{example}{Example}

\newtheorem*{defn}{Definition}

\newcommand{\spec}{\mathrm{spec}}

\title[Monoid spectra]{A topological description of the space of prime ideals of a monoid}
\author{Richard Vale}
\address{Department of Mathematics, Cornell University, Malott Hall, Ithaca N.Y. 14853-4201, United States}
\email{rv83@cornell.edu}
\begin{document}
\begin{abstract}
We describe which topological spaces can arise as the prime spectrum of a commutative monoid, in the spirit of Hochster's and Brenner's theses.
\end{abstract}
\maketitle
\section{Introduction}
\subsection{}
In this note, all monoids are commutative and the unit element of a monoid $M$ is written as $1$. Associated to a commutative monoid is a space called the prime spectrum. We will refer to it as the \emph{Kato spectrum} in order to avoid confusion with the prime spectrum of a commutative ring.
\begin{defn}
Let $M$ be a monoid. A subset $I \subset M$ is called an \emph{ideal} if for all $x \in I$ and $m \in M$, we have $xm \in I$. An ideal $p$ is \emph{prime} if $M \setminus p$ is a submonoid of $M$ (in particular, $1 \in M \setminus p$.)

We define $\spec(M)$ to be the set of all prime ideals of $M$. For $f \in M$, we define
$$D(f)=\{ p \in \spec(M): f \notin p\}.$$
The \emph{Kato spectrum} of $M$ is the set $\spec(M)$ with the topology with a base given by the $D(f)$.
\end{defn}
The Kato spectrum can also be equipped with a sheaf of monoids, but in this paper we are concerned only with the underlying space.
\subsection{}
The notion of prime ideal in a commutative semigroup goes ``back to antiquity" according to \cite{Grillet}. The Zariski topology on the set of prime ideals goes back at least to Kist \cite{Kist}.
The Kato spectrum was introduced by Kato \cite{Kato} in the study of toric singularities. It was later used by Deitmar \cite{Deitmar} to construct a theory of ``schemes over the field with one element". 

If $R$ is a commutative ring and $M$ denotes the underlying multiplicative monoid $(R,\cdot)$ of $R$ then $\spec(M)$ is the set of unions of prime ideals of $R$. This space has appeared in some constructions of spaces associated functorially to an arbitrary ring. See \cite{Aryapoor} and references therein.
\subsection{}\label{Hochster}
The aim of the present paper is to prove an analogue of the following theorem of Hochster. In this paper, we say a topological space $X$ is \emph{compact} if every open cover of $X$ has a finite subcover (in algebraic geometry, this is often called quasi-compactness.)
\begin{thrm}\cite{Hochster}
A topological space $X$ is homeomorphic to $\mathrm{Spec}(R)$ for some commutative ring $R$ if and only if the following three properties hold.
\begin{itemize}
\item $X$ is $T_0$.
\item The set of open compact subsets of $X$ is a base of $X$ which contains $X$ and is closed under finite intersections.
\item Every irreducible closed subset of $X$ is the closure of a unique point.
\end{itemize}
\end{thrm}
Theorem \ref{Hochster} is useful in constructing examples of affine schemes whose underlying spaces behave pathologically. It also began the study of \emph{spectral spaces}.
\subsection{}\label{mainthm}
Various authors have generalized Hochster's Theorem to other kinds of spectrum. For example, Hochster \cite{Hochster2} proved an analogue of the theorem with $\mathrm{Spec}$ replaced by the set of minimal prime ideals of a ring and Echi \cite{Echi} proved an analogue for the Goldman spectrum of a ring. Also, Kist \cite{Kist} studied topological properties of the space of \emph{minimal} prime ideals in a commutative semigroup. Brenner \cite{Brenner} proved the following analogue of Hochster's Theorem for monoids.
\begin{defn}\label{blobdef}
Let $X$ be a topological space. A subset $A\subset X$ is called a \emph{blob} if there exists $a \in X$ such that $A$ is the intersection of all open subsets of $X$ which contain $a$.
\end{defn}
\begin{thrm}\cite[Satz 2.3.2]{Brenner}
A topological space $X$ is homeomorphic to $\spec(M)$ for some commutative monoid $M$ if and only if the following properties hold
\begin{itemize}
\item $X$ is $T_0$.
\item The set of open blobs of $X$ is a base of $X$ which contains $X$ and is closed under finite intersections.
\item Every intersection of irreducible closed subsets of $X$ is the closure of a unique point.
\end{itemize}
and $X$ also satisfies $(*)$ where
\begin{align*}
(*) &\text{ If } \{U_\lambda: \lambda \in \Lambda\} \text{ is a collection of open blobs in } X \text{ and } U \text{ is an open set }\\ &\text{ with }\bigcap_\lambda U_\lambda \subset U,
\text{ then there exist } \lambda_1, \ldots, \lambda_n \in \Lambda \text{ with } \bigcap_{i=1}^n U_{\lambda_i} \subset U.
\end{align*}
\end{thrm}
\subsection{}
In this note, we will show that the first three properties in Theorem \ref{mainthm} may be expressed by saying that $X$ possesses a certain algebraic structure, essentially coming from the fact that the set of prime ideals of a monoid is closed under union. This yields a nice description of those spaces which can be the prime spectrum of a finitely-generated monoid. See Corollary \ref{finitecase}.
\subsection{Schemes over $\mathbb{F}_1$}
Vezzani \cite{Vezzani} has recently shown that Deitmar's definition of a scheme over $\mathbb{F}_1$ is equivalent to another definition due to To\"{e}n and Vaiqu\'{e} \cite{TV}, in which schemes are constructed by glueing together objects of the opposite of the category of monoids via a general construction which mimics (and generalizes) the usual definition of a scheme. See \cite[Section 1.2]{Pena-Lorscheid} for a fuller description of this. Theorem \ref{mainthm} may therefore be useful in understanding the nature of the spaces which play the role of affine schemes over $\mathbb{F}_1$ in these theories.
\section{Exponentiation}
\subsection{}
In this section, we characterize those spaces which satisfy the first three conditions of Theorem \ref{mainthm}, in order to understand how the condition $(*)$ fits in.
\begin{defn}
A poset $(P, \le)$ is called a \emph{join semilattice} if every pair of elements $x,y \in P$ has a least upper bound $\sup\{x,y\} \in P$. A join semilattice $P$ is \emph{complete} if every $A \subset P$ has a least upper bound $\sup(A) \in P$.

By definition, a map $f: (P, \le_P) \rightarrow (Q, \le_Q)$ of complete join semilattices is an order-preserving function which also satisfies $f(\sup(A))=\sup(f(A))$ for all $A \subset P$.
\end{defn}
\begin{defn}
Let $X$ be a topological space. We call a base $B$ of open sets of $X$ \emph{monoidal} if $X \in B$ and $B$ is closed under finite intersections.
\end{defn}
\subsection{}
We define a category $\mathbb{M}$ as follows. Objects of $\mathbb{M}$ are pairs $(X,B)$ where $X$ is a $T_0$ space and $B$ is a monoidal base of $X$. A morphism $(X,B) \rightarrow (Y,C)$ is a function $f:X \rightarrow Y$ such that $f^{-1}(U) \in B$ for all $U \in C$.
\begin{defn}\label{mcjsdef}
We call an object $(X,B)$ of $\mathbb{M}$ an \emph{$\mathbb{M}$--complete join semilattice} if there is a partial order $\le$ on $X$ such that $(X, \le)$ is a complete join semilattice, and such that for all $A \subset X$ and all $U \in B$, we have
$$A \subset U \iff \sup(A) \in U.$$
A morphism of $\mathbb{M}$--complete join semilattices is a morphism in $\mathbb{M}$ which is also a map of complete join-semilattices (ie. preserves suprema.)
\end{defn}
\subsection{}\label{monoid example}
Recall that the \emph{specialization order} $\le$ on a $T_0$ space $X$ is defined by $x \le y$ if and only if for all open sets $U$, $y \in U \implies x \in U$. Alternatively, $x \le y$ if and only if $y \in \overline{\{x\}}$. One can check that if the partial order $\le$ makes $(X,B)$ into an $\mathbb{M}$--complete join semilattice, then in fact $\le$ must be the specialization order.
\begin{example}\begin{rm}
Let $M$ be a commutative monoid. Let $X=\spec(M)$ and let $B=\{D(f): f \in M\}$, a monoidal base of $X$. Define a partial order on $X$ by $p \le q$ if and only if $p \subset q$. Then $(X,B)$ is an $\mathbb{M}$--complete join semilattice, because if $A \subset X$, we may take $\sup(A)=\bigcup_{p \in A} p$, which is a prime ideal of $M$.
\end{rm}\end{example}
We aim to show that each object $(X,B)$ of $\mathbb{M}$ can be completed to an $\mathbb{M}$--complete join semilattice. 
\subsection{The exponential}\label{expdef}
We define a functor $E: \mathbb{M} \rightarrow\mathbb{M}$ as follows. Let $(X,B)$ be an object of $\mathbb{M}$. Let $\mathcal{P}(X)$ be the power set of $X$. Let $B_0=\{\mathcal{P}(U): U \in B\}$. Then $B_0$ is a base for a topology on $\mathcal{P}(X)$. 
The resulting topological space $\mathcal{P}(X)$ may not be $T_0$, but we define an equivalence relation on $\mathcal{P}(X)$ by $A \sim B$ if and only if $A$ and $B$ belong to the same open sets of $\mathcal{P}(X)$.
\begin{defn}
The \emph{exponential} $E(X,B)$ is the quotient space $\mathcal{P}(X)/\sim$ equipped with the monoidal base $\widetilde{B}=\{\widetilde{U}: U \in B\}$ where for $U \in B$, we define
$$\widetilde{U} = \{[A] \in \mathcal{P}(X)/\sim \text{ such that } A \subset U\}.$$
Here, we write $[A]$ for the equivalence class in $\mathcal{P}(X)/\sim$ of $A \in \mathcal{P}(X)$.
\end{defn}
\begin{example}
\begin{rm}
Let $X=\{x\}$ be a one-point space. Let $B=\{\varnothing, X\}$, a monoidal base of $X$. Then $\mathcal{P}(X)=\{\varnothing, X\}$ with base $\{\{\varnothing\}, \mathcal{P}(X)\}$. This space is $T_0$, and so $E(X,B)= \mathcal{P}(X)$ is the two-point Sierpinski space. If we had taken the base $B'=\{X\}$ instead, we would get $E(X,B') \cong (X,B')$. Thus, $E(X,B)$ depends on the choice of the base $B$.
\end{rm}
\end{example}
\subsection{}
The following example was the original motivation for defining $E$.
\begin{example}
\begin{rm}
Let $R$ be a commutative ring and let $X=\mathrm{Spec}(R)$. For $f \in R$, let $D(f)=\{p \in \mathrm{Spec}(R): f \notin p\}$. Let $B$ be the base of $X$ consisting of all the $D(f)$. Let $M= (R,\cdot)$ be the underlying multiplicative monoid of $R$. Then $E(X,B) \cong \spec(M)$.
\end{rm}
\end{example}
\subsection{}
It is easy to see that $E: \mathbb{M} \rightarrow \mathbb{M}$ is a functor and that if $(X,B)$ is an object of $\mathbb{M}$ then there is a map $i=i_{(X,B)} : (X,B) \rightarrow E(X,B)$ defined by $i(x)=[\{x\}]$, the equivalence class of the singleton $\{x\}$. The map $i_{(X,B)}$ is injective and its image is dense in $E(X,B)\setminus\{[\varnothing]\}$. The $i_{(X,B)}$ define a natural transformation $id \rightarrow E$ of functors $\mathbb{M} \rightarrow \mathbb{M}$.
\subsection{}
If $(X,B)$ is an object of $\mathbb{M}$ then $E(X,B)$ is an $\mathbb{M}$--complete join semilattice under the specialization order. Indeed, the reader can check that if $A_\lambda$, $\lambda \in \Lambda$ are subsets of $X$, then the supremum of $\{[A_\lambda]: \lambda \in \Lambda\}$ is $\left[\bigcup_{\lambda \in \Lambda} A_\lambda\right]$.
\subsection{}\label{Euniversal}
We now show that $E(X,B)$ is the smallest $\mathbb{M}$--complete join semilattice which contains $(X,B)$.
\begin{prop}
Let $(X,B)$ be an object of $\mathbb{M}$.

The natural map $i :(X,B) \rightarrow E(X,B)$ is the universal map from $(X,B)$ to an $\mathbb{M}$--complete join-semilattice in the following sense:

If $\theta: (X,B) \rightarrow (Y,C)$ is a map in $\mathbb{M}$ and $(Y,C)$ is an $\mathbb{M}$--complete join-semilattice, then there exists a unique map of $\mathbb{M}$--complete join semilattices $\widehat{\theta}: E(X,B) \rightarrow (Y,C)$ such that the following diagram commutes.
$$\xymatrix{
(X,B) \ar[d]^i \ar[r]^\theta & (Y,C)\\
E(X,B) \ar[ur]_{\widehat{\theta}} &
}$$
\end{prop}
\begin{proof}
The map $\widehat{\theta}$ is defined by $\widehat{\theta}([A])=\sup(\theta(A))$ for all $A \subset X$. It is routine to check that $\widehat{\theta}$ is the unique morphism of $\mathbb{M}$--complete join semilattices which makes the diagram commute.
\end{proof}
\subsection{}\label{E^2=E}
We have the following corollary of Proposition \ref{Euniversal}.
\begin{cor}
$E^2=E$ as functors $\mathbb{M} \rightarrow \mathbb{M}$.
\end{cor}
\begin{proof}
This follows directly from Proposition \ref{Euniversal} if we note that if $(X,B)$ and $(Y,C)$ are $\mathbb{M}$--complete join semilattices and $f: (X,B) \rightarrow (Y,C)$ is a morphism in $\mathbb{M}$, then $f$ necessarily preserves suprema.
\end{proof}
\subsection{}\label{expcharacterization}
Now we give a topological characterization of which spaces can arise as $E(X,B)$ for some $(X,B)$. First, we need the following Lemma.
\begin{lem}\label{bloblem}
Let $(X,B)$ be an $\mathbb{M}$--complete join semilattice. Let $U \subset X$ be open. Then $U$ is a blob (see Definition \ref{blobdef}) if and only if $U \in B$.
\end{lem}
\begin{proof}
Suppose $U \subset X$ is an open blob. Then there exist $U_\lambda \in B$ with $U =\bigcup_\lambda U_\lambda$ because $B$ is a base for the topology on $X$. Since $U$ is a blob, there exists $a \in X$ such that $U$ is the intersection of all the open sets which contain $a$. Therefore, $a \in U_\lambda$ for some $\lambda$, and so $U \subset U_\lambda$. So $U=U_\lambda \in B$. Conversely, suppose $U \in B$. Then $\sup(U) \in U$ by Definition \ref{mcjsdef}. Since the order $\le$ on $X$ coincides with the specialization order (see Section \ref{monoid example}), for all $x \in X$ we have $x \in U$ if and only if $x \le \sup (U)$. Therefore, $U$ is the intersection of all the open sets which contain $\sup(U)$, so $U$ is a blob.
\end{proof}
\begin{thrm}
Let $X$ be a $T_0$ space. The following are equivalent.
\begin{enumerate}
\item The open blobs form a monoidal base of $X$, and for all $A \subset X$, $\bigcap_{x \in A} \overline{\{x\}}$ is the closure of a point.
\item If $B$ denotes the set of open blobs of $X$, then $(X,B)$ is an $\mathbb{M}$--complete join semilattice.
\item There exists a monoidal base $B$ of $X$ such that $(X,B)$ is an $\mathbb{M}$--complete join semilattice.
\item There exists a monoidal base $B$ of $X$ such that $(X,B) \cong E(X,B)$ in $\mathbb{M}$.
\item There exists a monoidal base $B$ of $X$ and an object $(Y,C)$ of $\mathbb{M}$ such that $(X,B) \cong E(Y,C)$ in $\mathbb{M}$.
\end{enumerate}
\end{thrm}
\begin{proof}
The equivalence of $(3)$, $(4)$ and $(5)$ follows directly from Proposition \ref{Euniversal} and Corollary \ref{E^2=E}. Also, $(2)$ trivially implies $(3)$. It remains to show that $(1)$ and $(3)$ are equivalent.

$(1) \implies (2)$: Let $B$ be the set of all open blobs of $X$. We claim that $(X,B)$ is an $\mathbb{M}$--complete join semilattice. Define $\le$ to be the specialization order on $X$, so that $x \le y$ if and only if $y \in \overline{\{x\}}$. For $A \subset X$, there is a point $y$ such that
$\bigcap_{a \in A}\overline{\{a\}}=\overline{\{y\}}$, and we see that $y$ is the supremum of $A$ in the ordering $\le$. So $(X, \le)$ is a complete join-semilattice. Now let $U$ be an open blob and $A \subset X$. We must show that $A \subset U$ if and only if $\sup(A) \in U$. There is some $a \in X$ such that $U$ is the intersection of all open sets containing $a$. In other words, $U =\{x \in X: x \le a\}$. So if $A \subset U$ then $\sup(A) \in U$, while if $\sup(A) \in U$ then $A \subset U$ because $x \le \sup(A) \le a$ for all $x \in A$. Therefore, $(X,B)$ is an $\mathbb{M}$--complete join semilattice.

$(3) \implies (1)$:
Suppose $(X,B)$ is an $\mathbb{M}$--complete join semilattice.
%
By Lemma \ref{bloblem}, $B$ is precisely the set of open blobs, and so these form a monoidal base. Now let $A \subset X$. It is easy to check that $\bigcap_{a \in A}\overline{\{a\}} = \overline{\{\sup(A)\}}$, and so we are done.
\end{proof}
\section{Monoids}
\subsection{}\label{weakprop}
In this section, we give a proof of Theorem \ref{mainthm} following \cite{Brenner}.
\begin{prop}
Let $X$ be a topological space. The following are equivalent.
\begin{enumerate}
\item $X$ is homeomorphic to $\spec(M)$ for some monoid $M$.
\item $X$ is $T_0$ and there is a monoidal base $B$ of $X$ such that $(X,B)$ is an $\mathbb{M}$--complete join semilattice and the map
    \begin{align*}\varphi: X &\rightarrow \spec(B, \cap)\\
    x &\mapsto \{U \in B: x \notin U\}
    \end{align*}
    is a bijection.
\end{enumerate}
\end{prop}
\begin{proof}
Suppose $X=\spec(M)$. Then take $B=\{D(f): f \in M\}$. Then $B$ is a monoidal base of $X$ and $(X,B)$ is an $\mathbb{M}$--complete join semilattice. The map $\varphi:X \rightarrow \spec(B, \cap)$ is given by $p \mapsto \{D(g): p \notin D(g)\}=\{D(g): g \in p\}$. This has an inverse given by sending a prime ideal $Q$ of the monoid $(B, \cap)$ to the prime ideal $\{f \in M : D(f) \in Q\}$ of $M$.

Conversely, suppose $X$ has a monoidal base $B$ and the map $x \mapsto \{U \in B: x \notin U\}$ is bijective. We need only show that this map is in fact a homeomorphism. This follows from \cite[Lemma 4.3]{Kist}. Alternatively, we may define a function $\theta: \spec(B, \cap) \rightarrow X$ by $\theta(p) = \sup (\bigcap_{U \in B \setminus p} U)$. Then for all $V \in B$ and $x \in X$, $\theta \varphi(x)=\theta(\{U \in B: x \notin U\}) \in V$ if and only if $x \in V$. Therefore, $\theta\varphi(x)=x$ since $X$ is $T_0$. So $\theta$ is a one-sided inverse to $\varphi$, and therefore $\theta$ is inverse to $\varphi$ since $\varphi$ is a bijection. To show that $\theta$ is continuous, observe that $\varphi$ is an open map because for $V \in B$, $\varphi(V)=\{p \in \spec(B, \cap): V \notin p\}$.
\end{proof}
\subsection{}\label{weakprop2}
\begin{prop}
Let $X$ be a topological space. The following are equivalent.
\begin{enumerate}
\item $X$ is homeomorphic to $\spec(M)$ for some monoid $M$.
\item $X$ is $T_0$ and there is a monoidal base $B$ of $X$ such that $(X,B)$ is an $\mathbb{M}$--complete join semilattice and the condition $(*)$ of Theorem \ref{mainthm} holds for $X$.
\end{enumerate}
\end{prop}
\begin{proof}
Suppose $X$ is homeomorphic to $\spec(M)$. Then $X$ is $T_0$. Let $B=\{D(f):f \in M\}$, a monoidal base of $X$. By Example \ref{monoid example}, $(X,B)$ is an $\mathbb{M}$--complete join semilattice. By
Lemma \ref{bloblem}, the open blobs of $X$ are precisely the elements of $B$. Suppose $U_\lambda$, $\lambda \in \Lambda$, are open blobs. Then
$$p=\{U \in B: U \text{ does not contain any finite intersection of the }
U_\lambda \}$$
is a prime ideal of the monoid $(B, \cap)$. By the proof of Proposition \ref{weakprop}, there exists $x \in M$ such that $p =\{U \in B: x \notin U\}$. In other words, for any $U \in B$, $x \in U$ if and only if $U$ contains some $\bigcap_{i=1}^n U_{\lambda_i}$. In particular, $x \in U_\lambda$ for all $\lambda$ and so $x \in\bigcap_\lambda U_\lambda$. Suppose $U$ is an open set and $U \supset \bigcap_\lambda U_\lambda$. Then $x \in U$ and so $x$ is contained in some open blob $U' \subset U$, since the open blobs are a base. So $U' \subset U$ contains some finite intersection of the $U_{\lambda}$. Thus, $(*)$ holds.

Conversely, if $X$ satisfies $(2)$, then by Proposition \ref{weakprop}, we just need to show that the function $\varphi:X \rightarrow \spec(B, \cap)$ defined by $\varphi(x)=\{U: x \notin U\}$ is a bijection. We define $\theta: \spec(B, \cap) \rightarrow X$ by $\theta (p)=\sup(\bigcap_{W \in B \setminus p} W)$. We have seen in the proof of Proposition \ref{weakprop} that $\theta \varphi$ is the identity, so we need to show that if $p$ is a prime ideal of $(B, \cap)$, then $\varphi\theta(p)=p$. We have $\varphi\theta(p)=\{U \in B: U \nsupseteq \bigcap_{W \in B\setminus p} W\}$. Thus, for $V \in B$, we see that $V \notin \varphi\theta(p)$ if and only if $V \supset \bigcap_{W \in B\setminus p} W$. By $(*)$ combined with Lemma \ref{bloblem}, this holds if and only if there are $W_1, \ldots, W_n \notin p$ with $V \supset \bigcap_{i=1}^n W_i$. Since $p$ is a prime ideal, this holds if and only if $V \notin p$. Thus, $\varphi\theta(p)=p$ as required.
\end{proof}
\subsection{}
We can also prove directly from the definition of $\spec(M)$ that $\spec(M)$ satisfies the condition $(*)$ of Theorem \ref{mainthm}. This can be done by first reducing to the case of an inclusion of basic open sets $\bigcap_{\lambda \in \Lambda} D(f_\lambda) \subset D(g)$ and then considering the prime ideal $p=M\setminus q$ where
$$q=\{h \in M: \exists t \in M, \lambda_1, \ldots, \lambda_n \in \Lambda, k_1, \ldots, k_n \ge 0 \text{ with } th=f_{\lambda_1}^{k_1} \cdots f_{\lambda_n}^{k_n}\}.$$
The details are left as an exercise.
\subsection{}\label{sober}
We are now ready to prove Theorem \ref{mainthm}. We use the following lemma.
\begin{lem}
Let $M$ be a monoid. Then every irreducible closed subset of $X=\spec(M)$ is the closure of a point.
\end{lem}
\begin{proof}
This is like the analogous fact from algebraic geometry. By \cite[Section 2]{Vezzani}, every closed subset of $X$ has the form $V(I)=\{p: p \supset I\}$ where $I$ is an intersection of prime ideals of $M$. If $V(I)$ is irreducible and $x, y \notin I$, then $V(I) \cap D(x) \neq \varnothing$ and $V(I) \cap D(y) \neq \varnothing$, and so $V(I) \cap D(x) \cap D(y)=V(I) \cap D(xy) \neq \varnothing$ by irreducibility. Therefore, $xy \notin I$ and so $I$ is prime, Similarly, if $I$ is prime then $V(I)$ is irreducible.

Thus, the irreducible closed sets of $X$ are precisely the sets $V(p)$ for $p \in X$. But $V(p)=\overline{\{p\}}$, as required.
\end{proof}
\subsection{Proof of Theorem \ref{mainthm}}
Let $X=\spec(M)$ for a monoid $M$. By Proposition \ref{weakprop2}, there is a monoidal base of $X$ such that $(X,B)$ is an $\mathbb{M}$--complete join semilattice. By Theorem \ref{expcharacterization}, the open blobs of $X$ form a monoidal base of $X$. By Lemma \ref{sober}, every closed irreducible subset of $X$ is the closure of a point. By Theorem \ref{expcharacterization} again, every intersection of closed irreducible subsets of $X$ is the closure of a point. This point is unique because $X$ is a $T_0$ space. By Theorem \ref{weakprop2}, the condition $(*)$ of Theorem \ref{mainthm} holds. Thus, all the conditions of Theorem \ref{mainthm} are satisfied by $X$.

Conversely, suppose $X$ satisfies the conditions of Theorem \ref{mainthm}. Let $B$ be the set of open blobs of $X$. Since $\overline{\{x\}}$ is an irreducible closed set for any $x \in X$, condition $(1)$ of Theorem \ref{expcharacterization} holds for $X$. By Theorem \ref{expcharacterization}, $(X,B)$ is an $\mathbb{M}$--complete join semilattice. Since $X$ satisfies $(*)$, we get from Proposition \ref{weakprop2} that $X$ is homeomorphic to $\spec(M)$ for some $M$.

This completes the proof of Theorem \ref{mainthm}.
\section{Remarks and applications}
\subsection{}
The proof of Theorem \ref{mainthm} 
implies that if $M$ is a monoid then $\spec(M)$ is homeomorphic to $\spec(B, \cap)$ where $B = \{D(f): f \in M\}$. The monoid $(B, \cap)$ is known as the \emph{universal semilattice} of $M$ (\cite[III.1.4]{Grillet}). Thus, every monoid spectrum is the spectrum of a semilattice. From this, we get a second characterization of monoid spectra as those spaces which can be obtained as the soberification (space of closed irreducible subsets) of some meet-semilattice $(P, \le)$ with greatest element, equipped with the topology whose open sets are the lower order ideals.
\subsection{The necessity of $(*)$}
The condition $(*)$ of Theorem \ref{mainthm} is ugly-looking. However, it is not possible to remove it (or some equivalent condition). To see this, we must exhibit a space which satisfies the first three conditions of Theorem \ref{mainthm} but is not homeomorphic to $\spec(M)$ for any monoid $M$. In view of Theorem \ref{expcharacterization}, we just need to exhibit a space $(X,B)$ such that $E(X,B)$ does not satisfy $(*)$.
\begin{example}
\begin{rm}
Let $X$ be a space with a monoidal base $B$ such that there are $U_n \in B$ with $\bigcap_{n \in \mathbb{N}} U_n = \varnothing$, but no finite intersection of the $U_n$ is empty. For example, take $X=\mathbb{R}$ with $B$ the collection of all open subsets of $\mathbb{R}$, and let $U_n =(0,\frac{1}{n})$ for $n \in \mathbb{N}$. Then from Definition \ref{expdef}, we see that $\bigcap_n \widetilde{U_n}= \{[A] \in \mathcal{P}(X)/\sim \text{ such that } A \subset \bigcap_n U_n \} =\{[\varnothing]\}$. Thus, $\bigcap_n \widetilde{U_n} \subset \widetilde{\varnothing}$. But no finite intersection of the $\widetilde{U_n}$ is contained in $\widetilde{\varnothing}$, and thus $(*)$ does not hold for $E(X,B)$.
\end{rm}
\end{example}
\subsection{}
A topological space $X$ is \emph{artinian} if every descending chain $U_1 \supset U_2 \supset \cdots$ of open sets terminates. Any artinian space (or even a space whose open blobs satisfy the descending chain condition) automatically satisfies $(*)$. Combining this observation with \cite[Proposition 5.5]{Kato}, we have the following corollary of Theorem \ref{mainthm}.
\begin{cor}\label{finitecase}
Let $X$ be a topological space. The following are equivalent.
\begin{enumerate}
\item $X$ is homeomorphic to $\spec(M)$ for some finite monoid $M$.
\item $X$ is homeomorphic to $\spec(M)$ for some finitely-generated monoid $M$.
\item $X$ is the underlying space of $E(Y,B)$ for some finite space $Y$ with a monoidal base $B$.
\end{enumerate}
\end{cor}
\subsection{}
Comparing Theorem \ref{mainthm} to Theorem \ref{Hochster}, and noting that the compact open subsets of $\spec(M)$ are just the finite unions of sets of the form $D(f)$, we observe that the underlying space of the Kato spectrum is always the underlying space of the spectrum of some commutative ring.
\subsection{Questions}
Theorem \ref{mainthm} gives rise to some natural questions.
\begin{itemize}
\item Is there a better way to express the condition $(*)$ of Theorem \ref{mainthm}?
\item Is there a characterization of spaces of the form $\spec(M)$ analogous to Hochster's characterization (\cite[Proposition 10]{Hochster}) of the underlying spaces of ring spectra as projective limits of finite $T_0$ spaces?
\end{itemize}

\end{document}